\documentclass[a4paper,twoside,12pt]{article}
\usepackage{amssymb,amsmath,amsthm,latexsym}
\usepackage{amsfonts}
\usepackage{graphicx,caption,subcaption}
\usepackage[pdftex,bookmarks,colorlinks=false]{hyperref}
\usepackage[hmargin=1.2in,vmargin=1.2in]{geometry}
\usepackage[mathscr]{euscript}
\usepackage[auth-sc-lg]{authblk}
\usepackage{setspace}
\def\ni{\noindent}
\def\N{\mathbb{N}_0}
\def\S {\Sigma}
\def\s {\sigma}
\def\cP{\mathcal{P}}

\newtheorem{thm}{Theorem}[section]

\newtheorem{defn}{Definition}[section]

\newtheorem{rem}{Remark}[section]

\newtheorem{prob}{Problem}

\title{\textbf{\sc Some New Results on Integer Additive Set-Valued Signed Graphs}}

\author{N. K. Sudev}
\affil{\small Department of Mathematics\\ Vidya Academy of Science \& Technology \\Thalakkottukara, Thrissur,  India.\\ {\tt sudevnk@gmail.com}}

\author{P. K. Ashraf}
\affil{\small Department of Mathematics\\Government Arts \& Science College\\ Kondotty, Malappuram, India\\ {\tt ashrafkalanthod@gmail.com}}

\author{K. A. Germina}
\affil{\small Department of Mathematics\\ University of Botswana\\Gaborone, Botswana.\\ {\tt srgerminaka@gmail.com}}

\date{}

\begin{document}
\maketitle

\begin{abstract}
Let $X$ denotes a set of non-negative integers and $\cP(X)$ be its power set. An integer additive set-labeling (IASL) of a graph $G$ is an injective set-valued function $f:V(G)\to \cP(X)-\{\emptyset\}$ such that the induced function $f^+:E(G) \to \cP(X)-\{\emptyset\}$ is defined by $f^+(uv)=f(u)+f(v);\ \forall\, uv\in E(G)$, where $f(u)+f(v)$ is the sumset of $f(u)$ and $f(v)$. An IASL of a signed graph is an IASL of its underlying graph $G$ together with the signature $\s$ defined by $\s(uv)=(-1)^{|f^+(uv)|};\ \forall\, uv\in E(\S)$. In this paper, we discuss certain characteristics of the signed graphs which admits certain types of integer additive set-labelings.
\end{abstract}

\vspace{0.2cm}

\ni \textbf{Key words}: Signed graphs; balanced signed graphs; clustering of signed graphs; integer additive set-labeled signed graphs; arithmetic integer additive set-labeled signed graphs.
\vspace{0.04in}


\section{Introduction}

For all  terms and definitions, not defined specifically in this paper, we refer to \cite{JAG,FH, DBW} and for the terminology and results in the theory of signed graphs, see \cite{AACE,GSH,FHS,TZ1,TZ2}. Unless mentioned otherwise, all graphs considered here are simple, finite and have no isolated vertices. 

A set-labeling of a graph $G$ can generally be considered as an assignment of the vertices of a graph to the subsets of a non-empty set $X$ in an injective manner such that the set-labels of the edges of $G$ are obtained by taking the symmetric difference the set-labels of their end vertices.  Some studies on set-labeling of signed graphs has been discussed in \cite{AGS1} and that of signed digraphs has been done in \cite{BDAS}. 

The {\em sum set} (see \cite{MBN}) of two sets $A$ and $B$, denoted by $A+B$, is defined as $A+B=\{a+b:a\in A, b\in B\}$.  Let $\mathbb{N}_0$ be the set of all non-negative integers and let $X$ be a non-empty subset of $\N$. Using the concepts of sumsets, we have the following notions as defined in \cite{GS1,GS0}.

\begin{defn}{\rm 
\cite{GS1,GS0} An {\em integer additive set-labeling} (IASL, in short) is an injective function $f:V(G)\to \cP(X)-\{\emptyset\}$ such that the induced function $f^+:E(G)\to \cP(X)-\{\emptyset\}$ is defined by $f^+{(uv)}=f(u)+f(v)~ \forall uv\in E(G)$.  A graph $G$ which admits an IASL is called an {\em integer additive set-labeled graph} (IASL-graph).}
\end{defn}  

\begin{defn}{\rm
\cite{GS1,GS0}  An {\em integer additive set-indexer} (IASI) is an injective function $f:V(G)\to \cP(X)-\{\emptyset\}$ such that the induced function $f^+:E(G) \to \cP(X)-\{\emptyset\}$ is also injective. A graph $G$ which admits an IASI is called an {\em integer additive set-indexed graph} (IASI-graph).}
\end{defn}   

The cardinality of the set-label of an element (vertex or edge) of a graph $G$ is called the {\em set-indexing number} of that element.

\subsection{Signed Graphs and Their IASLs}

A \textit{signed graph} (see \cite{TZ1,TZ2}), denoted by $\S(G,\s)$,  is a graph $G(V,E)$ together with a function $\s:E(G)\to \{+,-\}$ that assigns a sign, either $+$ or $-$, to each ordinary edge in $G$. The function $\s$ is called the {\em signature} or {\em sign function} of $\S$, which is defined on all edges except half edges and is required to be positive on free loops.

An edge $e$ of a signed graph $\S$ is said to be a \textit{positive edge} if $\s(e)=+$ and an edge $\s(e)$ of a signed graph $\S$ is said to be a \textit{negative edge} if $\s(e)=-$.

A  simple  cycle (or path) of a signed graph $\S$  is said to be {\em balanced} (see \cite{AACE,FHS}) if the product of signs of its edges is $+$. A  signed  graph  is said to be a {\em balanced signed graph} if it contains no half edges and all of its simple cycles are balanced. 

It is to be noted that the number of all negative  signed graph is balanced if and only if it is bipartite.  The notion of a signed graph that is associated to a given IASL-graph has been introduced as follows.

\begin{defn}{\rm 
\cite{GS5} Let $\S$ be a signed graph, with underlying graph $G$ and signature $\s$. An injective function $f:V(\S)\to \cP(X)-\{\emptyset\}$ is said to be an \textit{integer additive set-labeling} (IASL) of $\S$ if $f$ is an integer additive set-labeling of the underlying graph $G$ and the signature of $\S$ is defined by $\s(uv)=(-1)^{|f^+(uv)|}$. }
\end{defn}

A signed graph which admits an integer additive set-labeling is called an \textit{integer additive set-labeled signed graph} (IASL-signed graph) and is denoted by $\S_f$. If the context is clear, we can represent an IASL-signed graph by $\S$ itself. 

Some interesting studies on the signed graphs which admit different types of integer additive set-labelings have been done in \cite{GS5}. Motivated by the above mentioned studies, in this paper, we further study the characteristics of signed graphs which admits certain other IASLs.

\section{Arithmetic IASL-Signed Graphs}

By saying that a set is an arithmetic progression, we mean that the elements of that set is in arithmetic progression. The notion of an arithmetic integer additive set-labeling (AIASL) of a given graph  was introduced in \cite{GS3} as given below.  

\begin{defn}{\rm
\cite{GS3} An \textit{arithmetic integer additive set-labeling} of a graph $G$ is an integer additive set-indexer $f$ of $G$, with respect to which the set-labels of all vertices and edges of $G$ are arithmetic progressions. A graph that admits an AIASL is called an {\em arithmetic integer additive set-labeled graph}(AIASL-graph).}
\end{defn}

If the context is clear, the common difference of the set-label of an element of $G$ can be called the common difference of that element. The \textit{deterministic ratio} of an edge of $G$ is the ratio, $k\ge 1$ between the common differences of its end vertices. The following theorem is a necessary and sufficient condition for a graph $G$ to admit an arithmetic integer additive set-labeling.

\begin{thm}\label{Thm-AIASI-1}
{\rm \cite{GS3}} A graph $G$ admits an arithmetic integer additive set-labeling $f$ if and only if for every edge of $G$, the set-labels of its end vertices are arithmetic progressions with common differences $d_u$ and $d_v$ such that $d_u\le d_v$ and its deterministic ratio $\frac{d_v}{d_u}$ is a positive integer less than or equal to $|f(u)|$.
\end{thm}

The following theorem discussed the cardinality of the set-label of the edges of an arithmetic integer additive set-labeled graph.

\begin{thm}{\rm 
{\rm \cite{GS3}}  Let $f$ be an integer additive set-labeling of a graph $G$. Then, the set-indexing number of any edge $uv$ of $G$, with $d_u\le d_v$, is given by $|f(u)|+\frac{d_v}{d_u}\left(|f(u)|-1\right)$.	}
\end{thm}

All set-labels mentioned in this section are arithmetic progressions so that the given signed graph $\S$ admits an arithmetic integer additive set-labeling. Invoking the above results, we establish the following theorem for an edges of an AIASL-graph to be a positive edge.

\begin{thm}\label{Thm-2.1}
Let $\S$ be a signed graph which admits an AIASL $f$. Then, an edge $uv$ is a positive edge of an IASL-signed graph if and only if 
\begin{enumerate}\itemsep0mm
	\item[(i)] the set-labels $f(u)$ and $f(v)$ are of different parity, provided the deterministic ratio of the edge $uv$ is odd.
	\item[(ii)] the set-label of the end vertex, with minimum common difference, is of even parity, provided the deterministic ratio of the edge $uv$ is even.
\end{enumerate}
\end{thm}
\begin{proof}
By Theorem \ref{Thm-2.1}, the set-indexing number of $uv$ is given by $|f(u)|+k(|f(v)|-1)$. Note that the edge $uv$ is a positive edge of $\S$ if and only if the value of $|f(u)|+k(|f(v)|-1)$ is a positive integer.	
	
\textit{Case-1} Assume that the deterministic ratio of an edge $uv$ of $\S$ is $k=\frac{d_v}{d_u}$, an odd positive integer.  

Let $f(u)$ and $f(v)$ be of different parity. If $f(u)$ is of even parity and $f(v)$ is of odd parity, then $|f(v)|-1$ is even and hence $k(|f(v)|-1)$ is an even integer. Therefore, $(|f(u)|+k(|f(v)|-1))$ is an even integer and hence the signature of the edge $uv$ is positive. If $f(u)$ is of odd parity and $f(v)$ is of even parity, then $|f(v)|-1$ and hence $k(|f(v)|-1)$ are odd integers. Hence, $(|f(u)|+k(|f(v)|-1))$ is even and hence $\s(uv)$ is positive. 

Next, assume that $f(u)$ and $f(v)$ are of same parity. If $f(u)$ and $f(v)$ are of odd parity, then $|f(v)|-1$ is even and hence $k(|f(v)|-1)$ is also an even integer.  Therefore, $(|f(u)|+k(|f(v)|-1))$ is odd and $\s(uv)$ is negative. If $f(u)$ and $f(v)$ are of even parity, then $|f(v)|-1$ and hence $k(|f(v)|-1)$ are odd integers.  Therefore, $(|f(u)|+k(|f(v)|-1))$ is odd and $uv$ is a negative edge of $G$. 

Hence, any edge of an AIASL-signed graph $\S$ is a positive edge if and only if the set-labels of its end vertices are of the different parity.

\textit{Case-2}: Let the deterministic ratio $k$ of the edge $uv$ is an even. Then, $k(|f(v)|-1)$ is always even irrespective of the parity of the set-label $f(v)$. Hence, $(|f(u)|+k(|f(v)|-1))$ is even only when $|f(u)|$ is even. Therefore, $\s(uv)$ is positive if and only if $|f(u)|$ is even.
\end{proof}

The following theorem establishes a necessary and sufficient condition for an AIASL-signed graph to a balanced signed graph.

\begin{thm}\label{Thm-2.2}
An AIASL-signed graph $\S$ is balanced if and only if its underlying graph $G$ is bipartite.
\end{thm}
\begin{proof}
First assume that the underlying graph $G$ of a given signed graph is bipartite. Let $(V_1,V_2)$ be a bipartition of $G$. Then, label the vertices of $G$ in such a way that all vertices in the same partition have the same parity set-labels (arithmetic progressions). Here note that all edges in $\S$ now have either positive signature simultaneously or negative signature simultaneously. Since all cycles in $\S$ are even cycles, in all possible cases, the number of negative edges in each cycle will always be even. Hence, $\S$ is balanced.

Conversely, assume that the AIASL-signed graph is balanced. If possible, let the underlying graph $G$ be non-bipartite. Then, $G$ contains some odd cycles. Let $C: v_1v_2v_3\ldots v_nv_1$ be one odd cycle in $G$. Label the vertices of $G$ by certain arithmetic progressions in such a way that there exists minimum number of different parity set-labels. This can be done by labeling any two adjacent vertices by different parity sets. In this way, we can see without loss of generality that the vertices $v_1, v_3, v_5, \ldots, v_{n-2}$ have the same parity set-labels, say odd parity set-labels. Hence, the vertices $v_2, v_4, v_6, \ldots, v_{n-1}$ have even parity set-labels. If $v_n$ has an odd parity set-label, then the edge $v_nv_1$ is the one only negative edge of $G$ and if $v_n$ has an even parity set-label, then the edge $v_{n-1}v_n$ is the one only negative edge of $G$. Moreover, if the parity of any one vertex, say $v_i$ is changed, then the set-labels of vertices $v_{i-1}, v_i, v_{i+1}$ become same parity sets and hence the signature of two edges $v_{i-1}v_i$ and $v_iv_{i+1}$ become negative, keeping the number of negative edges in $\S$ odd. This is a contradiction to the hypothesis that $\S$ is balanced. Hence, $G$ must be bipartite.   
\end{proof}



\section{AIASL of Certain Associated Signed Graphs}

An IASL of signed graph $\S$ is said to \textit{induce} the IASL $f$ of $\S$ on its associated graphs if the following general conditions are hold.
\begin{enumerate}\itemsep0mm
	\item[(i)] the set-labels of corresponding elements of the graph and the associated graphs have the same set-labels,
	\item[(ii)] If a new edge (which is not in $\S$) is introduced in the associated graph, then the set-label and signature of that edge is determined by corresponding rules,  
	\item[(iii)] if one edge of $\S$ is replaced by another vertex (not in $\S$) in the associated graph, then the new vertex is given the same set-label of the removed edge.
\end{enumerate}

In this section, we discuss the induced characteristics of certain signed graphs associated with the AIASL-signed graphs. 

\begin{rem}{\rm 
Let $\S'$ be a signed subgraph of a balanced AIASL-signed graph $\S$ given by $\S'=\S-v$, where $v$ is an arbitrary vertex of $\S$. If $v$ is not in any cycle of $\S$, then the removal of $v$ does not affect the number negative edges in the cycles of $\S$. If $v$ is in a cycle $C$ of $\S$, then the removal of $v$ makes $C-v$ acyclic. In this case also, removal of $v$ does not affect the number negative edges in the cycles of $\S-v$. Then, $\S'$ is balanced whenever $\S$ is balanced.
}\end{rem}

A \textit{spanned signed subgraph} $\S'$ of a signed graph $\S$ is a signed graph which preserves signature and whose underlying graph $H$ is a spanning subgraph of the underlying graph $G$ of $\S$. The following result is an obvious and immediate from the corresponding definition of the balanced signed graphs.

\begin{rem}{\rm 
Let $\S'$ be a spanned signed subgraph of a balanced AIASL-signed graph $\S$. Then, $\S'$ is balanced with respect to induced labeling and signature if and only if the following conditions are hold.
\begin{enumerate}
	\item[(i)] the set $E(\S\setminus \S')$ contains even number of negative edges in $\S$, if the signed graph $\S$ is edge disjoint.
	\item[(ii)] the set $E(\S\setminus \S')$ contains odd number of negative edges in $\S$ if some of the negative edges are common to two more cycles in $\S$.
\end{enumerate}
}\end{rem}

A \textit{subdivision signed graph} of a graph $G$ is the graph obtained by introducing a vertex to some or all edges of $G$. It is to be noted that the set-label of this newly introduced vertex is the same as that of the edge in $\S$ to which it is introduced. In view of this fact, the following theorem checks whether a signed graph obtained by subdividing some or all edges of a balanced AIASL-signed graph $\S$ to be balanced.

\begin{thm}
A signed graph $\S'$ obtained by subdividing an edge $e$ of a balanced AIASL-signed graph $\S$ is a balanced under induced set-labeling if and only if $e$ is a cut edge of $\S$.
\end{thm}
\begin{proof}
Let $\S$ be a balanced AIASL-signed graph and let $e=uv$ be an arbitrary edge of $\S$. If $e$ is a cut edge of $\S$, then it is not contained any cycle of $\S$. Hence, the cycles in $\S'$ correspond to the same cycles in $\S$. Therefore, subdividing the edge $e$ will not affect the number of negative edges in any cycle of $\S'$. Therefore, $\S'$ is balanced.

Now assume that a signed graph $\S'$ obtained by subdividing the edge of $\S$ is balanced under induced set-labeling.  If possible, let $e$ be not a cut-edge of $G$, Then it is contained in a cycle $C$ of $\S$. Now, introduce a new vertex $w$ into the edge $uv$. Then, the edge $uv$ will be removed and two edges $uw$ and $vw$ are introduced to $\S$ and the vertex $w$ has the same set-label of the edge $uv$. Here, we need to consider the following cases.

\textit{Case-1}: Let $u$ and $v$ have the same parity set-labels. Then, the edge $uv$ has an odd parity set-label in $\S$ and hence the new vertex $w$ in $\S'$ has an odd parity set-label and negative signature. Hence, we have the following subcases.

\textit{Subcase-1.1}: If both  $u$ and $v$ have odd parity set-labels, the edges $uw$ and $vw$ are negative edges in the corresponding cycle $C'$ in $\S'$. Therefore, $C'$ contains odd number of negative edges, a contradiction to our hypothesis. 

\textit{Subcase-1.2}: If both $u$ and $v$ have odd parity set-labels, the edges $uw$ and $vw$ are positive edges in the corresponding cycle $C'$ in $\S'$. Therefore, the number of negative edges in $C'$ is one less than that of the corresponding cycle $C$ in $\S$, a contradiction to our hypothesis that $\S'$ is balanced.

\textit{Case-2}: Let $u$ and $v$ have different parity set-labels. Then, the edge $uv$ has an even parity set-label in $\S$ and hence the new vertex $w$ in $\S'$ has an even parity set-label and positive signature. Without loss of generality, let $u$ has even parity set-label and $v$ has odd parity set label in $\S$. Then, the edges $uw$ is a negative edge and the edge $vw$ is a positive edge in $\S'$. It can be noted that the cycle $C'$ contains one negative edge more than that in the corresponding cycle $C$ in $\S$. It is also a contradiction to the hypothesis. In all possible cases, we get contradiction and hence the edge $e$ must be a cut edge of $\S$. This complete the proof. 
\end{proof}

A signed graph $\S'$ is said to be homeomorphic to another signed graph $\S$ if $\S'$ is obtained by removing a vertex $v$ with $d(v)=2$ and is not a vertex of any triangle in $\S$, and joining the two pendant vertices thus formed by a new edge. This operation is said to be an elementary transformation on $\S$. The following theorem discusses the balance of a signed graph that is homeomorphic to a given balanced AIASL-signed graph $\S$.

\begin{thm}
A signed graph obtained from a balanced AIASL-signed graph $\S$ by applying an elementary transformation on a suitable vertex $v$  of $\S$ is a balanced signed graph with respect to induced set-labeling if and only if the vertex $v$ is not in any cycle in $\S$.
\end{thm} 
\begin{proof}
Let $\S$ be a balanced AIASL-signed graph and let $v$ be any vertex of $\S$ with degree $2$ and is not in any triangle in $\S$. Also, let $\S'$ be a signed graph obtained from $\S$ by applying an elementary transformation on $v$. 

Since $d(v)=2$, it is adjacent to two vertices, say $u$ and $w$ in $\S$. If $v$ is not in any cycle in $\S$, then the number of negative edges in any cycle of $\S$ and hence the number of negative edges in the corresponding cycles in $\S'$ will not be affected by the elementary transformation on $v$. Therefore, $\S'$ is balanced.

Conversely, assume that $v$ is a vertex of a cycle $C$ in $\S$. Then we need to consider the following cases.

\textit{Case-1}: Let $u$ and $w$ have same parity set-labels.  Here we have the following subcases.

\textit{Subcase-1.1}: If the set-label of $v$ and the set-label of $u$ and $w$ are of the same parity, then the edges $uv$ and $vw$ are negative edges in $C$ of $\S$ and the edge $uw$ is a negative edge in the corresponding cycle $C'$ in $\S'$. Therefore, $C'$ contains one negative edge less than that of $C$ in $\S$. Hence, in this case, $\S'$ not balanced.

\textit{Subcase-1.2}: If the parity of the set-label of $v$ is different from that of the set-labels of $u$ and $w$, then the edges $uv$ and $vw$ are positive edges in the cycle $C$ of $\S$. But, the edge $uw$ is a negative edge in the corresponding cycle $C'$ of $\S'$. Hence, the cycle $C'$ contains one negative cycle more than that of the corresponding cycle $C$ in $\S$. Therefore, in this case also, $\S'$ not balanced.

\textit{Case-2}: Let $u$ and $w$ have different parity set-labels. Then, the set-label of $v$  and the set-label either $u$ or $v$ are of the same parity. Without loss of generality, let the set-labels of $u$ and $v$ be of the same parity. Then, the edge $uv$ is a negative edge and $vw$ is a positive edge in the cycle $C$ of $\S$. But, since $u$ and $v$ have different parity set-labels, the edge $uw$ is a positive edge in the corresponding cycle $C'$ of $\S'$. Hence, the cycle $C'$ contains one negative  cycle less than that of the corresponding cycle $C$ in $\S$. Therefore, $\S'$ not balanced. This completes the proof.
\end{proof}

\section{Conclusion}

In this paper, we discussed the characteristics and properties of the signed graphs which admits arithmetic integer additive set-labeling with a prime focus on the balance of these signed graphs. There are several open problems in this area. Some of the open problems that seem to be promising for further investigations are following.

\begin{prob}{\rm 
Discuss the $k$-clusterability of different types of IASL-signed graphs for $k>2$.}
\end{prob}

\begin{prob}{\rm 
Discuss the balance and $2$-clusterability and general $k$-clusterability of other types of signed  graphs which admit different types of arithmetic IASLs.}
\end{prob}

\begin{prob}{\rm 
Discuss the balance and $2$-clusterability and general $k$-clusterability of graceful, sequential and topological IASL-signed  graphs.}
\end{prob}

\begin{prob}{\rm 
Discuss the admissibility of AIASLs by the signed graphs obtained from the AIASL-signed graphs by finite number of edge contractions.}
\end{prob}

\begin{prob}{\rm 
Discuss the admissibility of AIASLs by the signed graphs whose underlying graphs are the line graphs and total graphs of the underlying graphes of certain  AIASL-signed graphs.}
\end{prob}

Further studies on other characteristics of signed graphs corresponding to different IASL-graphs are also interesting and challenging. All these facts highlight the scope for further studies in this area.


\begin{thebibliography}{20}

\bibitem{BDAS} B. D. Acharya, {\em Set-valuations of signed digraphs}, J. Combin. Inform. System Sci., {\bf 37}(2-4)(2012), 145-167.

\bibitem{AACE} J Akiyama, D.  Avis, V. Cha\'{v}tal and H. Era, {\em Balancing signed graphs}, Discrete App. Math., {\bf 3}(4)(1981), 227-233., DOI: 10.1016/0166-218X(81)90001-9.

\bibitem{AGS1} P. K. Ashraf, K. A. Germina and N. K. Sudev, {\em A study on set-valuations of signed graphs}, communicated.

\bibitem {JAG} J. A. Gallian, {\em A dynamic survey of graph labeling}, Electron. J. Combin., (2015), (\#DS-6). 

\bibitem {GSH} K. A. Germina and S. Hameed, {\em On signed paths, signed cycles and their energies}, Appl. Math. Sci., {\bf 4}(70)(2010), 3455 – 3466. 

\bibitem {GS1} K. A. Germina and N. K. Sudev,  {\em On weakly uniform integer additive set-indexers of graphs}, Int. Math. Forum, {\bf 8}(37)(2013), 1827-1834. DOI: 10.12988/imf.2013.310188.

\bibitem{FH}  F. Harary, {\bf Graph theory}, Addison-Wesley Publ. Co. Inc., 1969.

\bibitem{FHS} F. Harary, {\em On the notion of balance of a signed graph}, The Michigan Math. J., {\bf 2}(2)(1953), 143-146.

\bibitem {MBN} M. B. Nathanson,  {\bf Additive number theory, inverse problems and geometry of sumsets}, Springer, New York, 1996.

\bibitem{GS0} N. K. Sudev and K. A. Germina, {\em On integer additive set-indexers of graphs}, Int. J. Math. Sci.  Engg. Appl., {\bf 8}(2)(2014), 11-22.

\bibitem{GS3} N. K. Sudev and K. A. Germina, {\em A study on arithmetic integer additive set-indexers of graphs}, Carpathian Math. Publ., 2016, to appear.

\bibitem{GS4} N. K. Sudev and K. A. Germina, {\em On certain types of arithmetic integer additive set-indexers of graphs}, Discrete Math. Algorithms Appl., {\bf 7}(1)(2015),1-15., DOI: 10.1142/S1793830915500251.

\bibitem{GS5} N. K. Sudev and K. A. Germina, {\em A study on integer additive set-valuation of signed graphs}, Carpathian Math. Publ., {\bf 7}(2)(2015), 236-246., DOI:10.15330/cmp.7.2.236-246.

\bibitem{DBW} D. B. West, {\bf Introduction to graph theory}, Pearson Education Inc., 2001.

\bibitem{TZ1} T. Zaslavsky, {\em Signed graphs}, Discrete Appl. Math., {\bf 4}(1)(1982), 47-74., DOI: 10.1016/0166-218X(82)90033-6.

\bibitem{TZ2} T. Zaslavsky, {\em Signed graphs and geometry}, J. Combin. Inform. System Sci., {\bf 37}(2-4)(2012), 95-143.

\end{thebibliography}
\end{document}